\documentclass[11pt,a4paper]{amsart}

\usepackage{amsmath,amssymb,amsthm,amscd,verbatim, enumerate}
\usepackage{graphicx}
\usepackage{epsfig}
\usepackage{hyperref}

\newtheorem{thm}{Theorem}
\newtheorem{lem}[thm]{Lemma}

\newtheorem{claim}[thm]{Claim}

\newtheorem*{thm*}{Theorem}

\theoremstyle{definition}

\theoremstyle{remark}

\newcommand{\cC}{\mathcal{C}}
\newcommand{\cD}{\mathcal{D}}
\newcommand{\cR}{\mathcal{R}}

\newcommand{\cF}{\mathcal{F}}

\renewcommand{\le}{\leqslant}

\renewcommand{\ge}{\geqslant}

\begin{document}

\title[Small feedback vertex sets in planar digraphs]{Small feedback
  vertex sets\\ in planar digraphs}

\author{Louis Esperet} 
\author{Laetitia Lemoine} 
\author{Fr\'ed\'eric Maffray}
\address{Laboratoire G-SCOP (CNRS,
   Univ. Grenoble-Alpes), Grenoble, France}
\email{\{louis.esperet,laetitia.lemoine,frederic.maffray\}@grenoble-inp.fr}

\thanks{The authors are partially supported by ANR Project Stint
  (\textsc{anr-13-bs02-0007}), and LabEx PERSYVAL-Lab
  (\textsc{anr-11-labx-0025}).}

\date{}
\sloppy

\begin{abstract}
  Let $G$ be a directed planar graph on $n$ vertices, with no
  directed cycle of length less than $g\ge 4$. We prove that $G$
  contains a set $X$ of vertices such that $G-X$ has no directed
  cycle, and $|X|\le \tfrac{5n-5}9$ if $g=4$, $|X|\le \tfrac{2n-5}4$
  if $g=5$, and $|X|\le \tfrac{2n-6}{g}$ if $g\ge 6$. This improves
  recent results of Golowich and Rolnick.
\end{abstract}

\maketitle


A directed graph $G$ (or digraph, in short) is said to be
acyclic if it does not contain any directed cycle. The \emph{digirth} of a digraph $G$ is
the minimum length of a directed cycle in $G$ (if $G$ is acyclic, we
set its digirth to $+\infty$).
A \emph{feedback
  vertex set} in a digraph $G$ is a set $X$ of vertices such that
$G-X$ is acyclic, and the minimum size of such a set is
denoted by $\tau(G)$. In this short note, we study the maximum
$f_g(n)$ of $\tau(G)$ over all planar digraphs $G$ on $n$
vertices with digirth $g$.  Harutyunyan~\cite{Har11,HM15} conjectured
that $f_3(n)\le \tfrac{2n}5$ for all $n$. This conjecture was recently
refuted by Knauer, Valicov and Wenger~\cite{KVW16} who showed that
$f_g(n)\ge \tfrac{n-1}{g-1}$ for all $g\ge 3$ and infinitely many values of
$n$. On the other hand, Golowich and Rolnick~\cite{GR15} recently
proved that $f_4(n)\le \tfrac{7n}{12}$, $f_5(n)\le \tfrac{8n}{15}$, and $f_g(n)\le \tfrac{3n-6}{g}$
for all $g\ge 6$ and $n$. Harutyunyan and Mohar~\cite{HM15} proved
that the vertex set of every planar digraph of digirth at least 5 can
be partitioned into two acyclic subgraphs. This result was very recently extended to planar digraphs
of digirth 4 by Li and Mohar~\cite{LM16}, and therefore $f_4(n)\le
\tfrac{n}2$.

This
short note is devoted to the following result, which improves all
the previous upper bounds for $g \ge 5$ (although the improvement for
$g=5$ is rather minor). Due to the very recent result of Li and Mohar~\cite{LM16}, our
result for $g=4$ is not best possible (however its proof is of
independent interest and might lead to further improvements).

\begin{thm}\label{th:main}
For all $n\ge 3$ we have $f_4(n)\le \tfrac{5n-5}9$, $f_5(n)\le
\tfrac{2n-5}{4}$ and for all $g\ge 6$, $f_g(n)\le \tfrac{2n-6}{g}$.
\end{thm}

In a planar graph, the degree of a face $F$, denoted by $d(F)$, is the sum of
the lengths (number of edges) of the boundary walks of $F$. In the
proof of Theorem~\ref{th:main}, we will need the following two 
simple lemmas.

\begin{lem}\label{lem:1}
Let $H$ be a planar bipartite graph, with bipartition $(U,V)$, such
that all faces of $H$ have degree at least 4, and all vertices of $V$ have
degree at least 2. Then $H$ contains at most $2|U|-4$ faces of degree at least
6.
\end{lem}

\begin{proof}
Assume that $H$ has $n$ vertices, $m$ edges, $f$ faces, and $f_6$
faces of degree at least 6. Let $N$ be the sum of the degrees of the
faces of $H$, plus twice the sum of the degrees of the vertices of
$V$. Observe that $N=4m$, so, by Euler's formula, $N\le 4n+4f-8$. The sum of degrees
of the faces of $H$ is at least $4(f-f_6)+6f_6=4f+2f_6$, and since each vertex
of $V$ has degree at least 2, the sum of the degrees of the vertices
of $V$ is at least $2|V|$. Therefore, $4f+2f_6+4|V|\le 4n+4f-8$. It
follows that $f_6\le 2 |U|-4$, as desired.
\end{proof}

\begin{lem}\label{lem:2}
Let $G$ be a connected planar graph, and let $S=\{F_1,\ldots,F_k\}$ be a set of
$k$ faces of $G$, such that each $F_i$ is bounded by a cycle, and
these cycles are pairwise vertex-disjoint. Then $\sum_{F \not\in S}
(3d(F)-6)\ge \sum_{i=1}^k (3d(F_i)+6)-12$, where the first sum varies
over faces $F$ of $G$ not contained in $S$.
\end{lem}

\begin{proof}
Let $n$, $m$, and $f$ denote the number of vertices, edges, and faces
of $G$, respectively. It follows from Euler's formula that the sum of $3d(F)-6$ over all
faces of $G$ is equal to $6m-6f=6n-12\ge 6\sum_{i=1}^k
d(F_i)-12$. Therefore,
$\sum_{F\not\in S} (3d(F)-6)\ge 6\sum_{i=1}^k
d(F_i)-12 - \sum_{i=1}^k (3d(F_i)-6)=\sum_{i=1}^k (3d(F_i)+6)-12$, as desired.
\end{proof}

We are now able to prove Theorem~\ref{th:main}.

\bigskip

\noindent \emph{Proof of Theorem~\ref{th:main}.}
We prove the result by induction on $n\ge 3$.
Let $G$ be a planar digraph with $n$ vertices and digirth $g\ge 4$. We can assume
without loss of generality
that $G$ has no multiple arcs, since $g\ge 4$ and removing one arc from a
collection of multiple arcs with the same orientation does not change
the value of $\tau(G)$. We can also assume that $G$ is connected,
since otherwise we can consider each connected component of $G$ separately
and the result clearly follows from the induction (since $g\ge 4$,
connected components of at most 2 vertices are acyclic and can thus be left
aside). Finally, we can assume that $G$ contains a directed cycle,
since otherwise $\tau(G)=0\le
\min\{\tfrac{5n-5}9,\tfrac{2n-5}{4},\tfrac{2n-6}{g}\}$ (since $n\ge
3$).

\smallskip

Let $\cC$
be a maximum collection of arc-disjoint directed cycles in $G$. Note
that $\cC$ is non-empty.
Fix a planar embedding of $G$. For a given directed cycle $C$ of $\cC$, we
denote by $\overline{C}$ the closed region bounded by $C$, and by
 $\mathring{C}$ the interior of $\overline{C}$. It follows from classical uncrossing techniques (see~\cite{GW97}
for instance), that we can assume without loss of generality that
the directed cycles of $\cC$ are pairwise non-crossing, i.e. for any two
elements $C_1,C_2 \in \cC$, either $\mathring{C_1}$ and $\mathring{C_2}$
are disjoint, or one is contained in the other. We define the partial
order $\preceq$ on $\cC$ as follows: $C_1 \preceq C_2$ if and only if
$\mathring{C_1}\subseteq \mathring{C_2}$. Note that $\preceq$ naturally defines a
rooted forest $\cF$ with vertex set $\cC$: the roots of each of the
components of $\cF$ are the maximal elements of $\preceq$, and the children of
any given node $C\in \cF$ are the maximal elements $C' \preceq C$
distinct from $C$ (the
fact that $\cF$ is indeed a forest follows from the non-crossing
property of the elements of $\cC$).

Consider a node $C$ of $\cF$, and the children $C_1,\ldots,C_k$ of $C$
in $\cF$. We define the closed region $\cR_C=\overline{C}-\bigcup_{1\le
  i \le k} \mathring{C_i}$. Let $\phi_C$ be the sum of $3d(F)-6$, over all faces $F$ of  $G$ lying in
$\cR_C$.

\begin{claim}\label{cl:1} Let $C_0$ be a node of $\cF$ with children
$C_1,\ldots,C_k$. Then $\phi_{C_0}\ge \tfrac{3}2 (g-2)k+\tfrac32g$. Moreover, if
$g\ge 6$, then $\phi_{C_0}\ge \tfrac{3}2 (g-2)k+\tfrac32g+3$.
\end{claim}

Assume first that the cycles $C_0,\ldots,C_k$ are pairwise vertex-disjoint.
Then, it follows from Lemma~\ref{lem:2} that $\phi_{C_0}\ge
(k+1)(3g+6)-12$. Note that since $g\ge 4$, we have $(k+1)(3g+6)-12\ge
\tfrac{3}2 (g-2)k+\tfrac32g$. Moreover, if $g\ge 6$, $(k+1)(3g+6)-12\ge
\tfrac{3}2 (g-2)k+\tfrac32g+3$, as desired. As a
consequence, we can assume that two of the cycles $C_0,\ldots,C_k$
intersect, and in particular, $k\ge 1$.

\smallskip

Consider the following planar bipartite graph $H$: the vertices of the
first partite set of $H$ are the directed cycles
$C_0,C_1,\ldots,C_k$, the vertices of the second partite set of $H$
are the vertices of $G$ lying in at least two cycles among
$C_0,C_1,\ldots,C_k$, and there is an edge in $H$ between some cycle
$C_i$ and some vertex $v$ if and only if $v\in C_i$ in $G$ (see
Figure~\ref{fig:ex}). Observe
that $H$ has a natural planar embedding in which all internal faces
have degree at least 4. Since $k\ge 1$ and at least two of the
cycles $C_0,\ldots,C_k$
intersect, the outerface also has degree at least 4.
Note that the faces $F_1,\ldots,F_t$ of $H$ are in
one-to-one correspondence with the maximal subsets
$\cD_1,\ldots,\cD_t$ of $\cR_{C_0}$ whose interior is connected. Also note that each face of $G\cap \cR_{C_0}$ is in precisely
one region $\cD_i$ and each arc of $\bigcup_{i=0}^{k} C_i$
(i.e. each arc on the boundary of $\cR_{C_0}$) is on the boundary of
precisely one region $\cD_i$. For each region $\cD_i$, let $\ell_i$ be
the number of arcs on the boundary of $\cD_i$, and observe that
$\sum_{i=1}^{t}\ell_i=\sum_{j=0}^{k} |C_j|$. Let $\phi_{\cD_i}$ be the sum of $3d(F)-6$, over all
faces $F$ of $G$ lying in $\cD_i$. It follows from Lemma~\ref{lem:2} (applied with $k=1$) that $\phi_{\cD_i}\ge 3\ell_i-6$, and therefore
$\phi_{C_0}=\sum_{i=1}^{t}\phi_{\cD_i}\ge \sum_{i=1}^{t}(3\ell_i-6)$.

\begin{figure}[htbp]
\begin{center}
\includegraphics[scale=1.2]{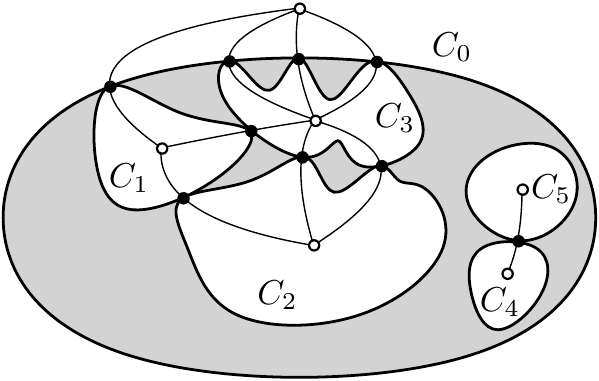}
\caption{The region $\cR_{C_0}$ (in gray) and the planar bipartite graph $H$. \label{fig:ex}}
\end{center}
\end{figure}

\smallskip

A region
$\cD_i$ with $\ell_i\ge 4$ is said to be of \emph{type 1}, and we set
$T_1=\{1\le i\le t \,|\, \cD_i \mbox{ is of type 1}\}$.
Since for any
$\ell \ge 4$ we have $3\ell-6\ge \tfrac{3\ell}2$, it follows from the
paragraph above that the regions $\cD_i$ of type 1 satisfy $\phi_{\cD_i}\ge \tfrac{3\ell_i}2$. Let $\cD_i$ be a region that is not of type 1. Since $G$ is simple,
$\ell_i=3$. Assume first that $\cD_i$ is bounded by (parts of) two
directed cycles of $\cC$ (in other words, $\cD_i$ corresponds to a face of
degree four in the graph $H$). In this case we say that $\cD_i$
is of \emph{type 2} and we set
$T_2=\{1\le i\le t \,|\, \cD_i \mbox{ is of type 2}\}$. Then the boundary of $\cD_i$
consists in two consecutive arcs $e_1,e_2$ of some directed cycle
$C^+$ of $\cC$, and one arc $e_3$ of some
directed cycle $C^-$ of $\cC$. Since $g\ge 4$, these three arcs do not form a
directed cycle, and therefore their orientation is transitive. It follows that
$|C^+|\ge g+1$, since otherwise the directed cycle obtained from $C^+$
by replacing $e_1,e_2$ with $e_3$ would have length $g-1$,
contradicting that $G$ has digirth at least $g$. Consequently,
$\sum_{i=0}^k |C_i|\ge (k+1)g+|T_2|$. If a region $\cD_i$ is not of
type 1 or 2, then $\ell_i=3$ and each of the 3 arcs on the boundary of
$\cD_i$ belongs to a different directed cycle of $\cC$. In other words, $\cD_i$
corresponds to some face of degree 6 in the graph $H$. Such a
region $\cD_i$ is said to be of \emph{type 3}, and we set
$T_3=\{1\le i\le t \,|\, \cD_i \mbox{ is of type 3}\}$. It follows from Lemma~\ref{lem:1} that the
number of faces of degree at least 6 in $H$ is at most
$2(k+1)-4$. Hence, we have $|T_3|\le 2k-2$. 

Using these bounds on
$|T_2|$ and $|T_3|$, together with the fact that for any $i\in T_2\cup
T_3$ we have $\phi_{\cD_i} \ge 3\ell_i-6=3=\tfrac{3\ell_i}2 -\tfrac{3}2$, we
obtain:

\begin{eqnarray*}
\phi_{C_0} & = & \sum_{i\in T_1} \phi_{\cD_i} + \sum_{i\in T_2}
                 \phi_{\cD_i} + \sum_{i\in T_3} \phi_{\cD_i} \\
& \ge & \sum_{i=1}^t  \tfrac{3\ell_i}2 - \tfrac32 |T_2| - \tfrac32
        |T_3|\\
& \ge & \tfrac32\, \sum_{i=0}^k  |C_i| - \tfrac32 |T_2| - \tfrac32
        (2k-2)\\
& \ge & \tfrac32 (k+1)g-3k+3 \, = \,  \tfrac32 (g-2)k+\tfrac32g+3, 
\end{eqnarray*}

as desired. This concludes the proof of Claim~\ref{cl:1}.\hfill $\Box$

\medskip

Let $C_1,\ldots,C_{k_\infty}$ be the $k_\infty$ maximal elements of
$\preceq$. 
We denote by $\cR_\infty$ the closed
region obtained from the plane by removing $\bigcup_{i=1}^{k_\infty} \mathring{C_i}$. Note that each face of
$G$ lies in precisely one of the regions $\cR_C$ ($C \in \cC$) or $\cR_\infty$.
Let $\phi_\infty$ be the sum of $3d(F)-6$, over all faces $F$ of  $G$ lying in
$R_\infty$. A proof similar to that of Claim~\ref{cl:1} shows that
$\phi_\infty\ge \tfrac32 k_\infty (g-2)+3$, and if $g\ge 6$, then $\phi_\infty\ge \tfrac32 k_\infty (g-2)+6$.

We now compute the sum $\phi$ of $3d(F)-6$ over all faces $F$ of
$G$. By Claim~\ref{cl:1},

\begin{eqnarray*}
\phi & =  & \phi_\infty + \sum_{C \in \cF}  \phi_C\\
& \ge &\tfrac32 k_\infty (g-2)+3 + (|\cC|-k_\infty) \tfrac32(g-2)+ |\cC|
        \cdot \tfrac32g \\
& \ge & (3g-3) |\cC|+3.
\end{eqnarray*}

If $g\ge 6$, a similar computation gives $\phi \ge 3g |\cC|+6 $.
On the other hand, it easily follows from Euler's formula that
$\phi=6n-12$. Therefore, $|\cC|\le \tfrac{2n-5}{g-1}$, and if $g\ge
6$, then $|\cC|\le \tfrac{2n-6}{g}$.

\medskip

Let $A$ be a set of arcs of $G$ of minimum size such that $G-A$ is acyclic.
It follows from the Lucchesi-Younger theorem~\cite{LY78} (see
also~\cite{GR15}) that $|A|=|\cC|$. Let $X$ be a set of vertices
covering the arcs of $A$, such that $X$ has minimum size. Then $G-X$
is acyclic. If $g=5$ we have $|X|\le
|A|=|\cC|\le \tfrac{2n-5}{4}$ and if $g\ge 6$, we have $|X|\le
|A|=|\cC|\le \tfrac{2n-6}{g}$, as desired. Assume now that $g=4$. In
this case $|A|=|\cC|\le \tfrac{2n-5}{3}$. It
was observed by Golowich and Rolnick~\cite{GR15} that $|X|\le
\tfrac13(n+|A|)$ (which easily follows from the fact that any graph on
$n$ vertices and $m$ edges contains an independent set of size at
least $\tfrac{2n}3-\tfrac{m}3$),
and thus, $|X|\le \tfrac{5n-5}{9}$.
This concludes the proof of Theorem~\ref{th:main}.\hfill $\Box$


 \section*{Final remark}

A natural problem is to determine the
precise value of $f_g(n)$, or at least its asymptotical value as $g$
tends to infinity. We
believe that $f_g(n)$ should be closer to the lower bound of $\tfrac{n-1}g$,
than to our upper bound of $\tfrac{2n-6}g$. 

For a digraph $G$, let $\tau^*(G)$ denote the the infimum real number $x$ for which
there are weights in $[0,1]$ on each vertex of $G$, summing up
to $x$, such that for each directed cycle $C$, the sum of the
weights of the vertices lying on $C$ is at least $1$. Goemans and Williamson~\cite{GW97} conjectured that for
any planar digraph $G$, $\tau(G)\le \tfrac32 \tau^*(G)$. If a planar
digraph $G$ on $n$ vertices has digirth at least $g$, then clearly $\tau^*(G)\le
\tfrac{n}{g}$ (this can be seen by assigning weight $1/g$ to
each vertex). Therefore, a direct consequence of the conjecture of
Goemans and Williamson would be that $f_g(n)\le \tfrac{3n}{2g}$.

\end{document}